\documentclass[a4wide,11pt]{scrartcl}

\usepackage[USenglish]{babel}
\usepackage[T1]{fontenc}
\usepackage[utf8]{inputenc}

\usepackage{graphicx,subfigure,tikz,pgfplots,epstopdf}
\graphicspath{{figs/}}

\usepackage{amsmath,amssymb,amsfonts,amsthm,array,geometry,stmaryrd,booktabs,paralist,todonotes,mathtools,pdflscape,siunitx,placeins}

\usepackage{ifpdf}
\ifpdf
\usepackage[pdftex,colorlinks=true,linkcolor=blue,citecolor=green,urlcolor=blue,bookmarks]{hyperref}
\else
\usepackage[dvipdfm,colorlinks=true,linkcolor=blue,citecolor=green,urlcolor=blue,bookmarks]{hyperref}
\fi 

\numberwithin{equation}{section}

\newcommand{\R}{\mathbb{R}}
\newcommand{\dd}{\textup{d}}

\newtheorem{thm}{Theorem}[section]

\newtheorem{rem}[thm]{Remark}
\newtheorem{cor}[thm]{Corollary}

\usepackage{url}
\usepackage{dsfont}
\usepackage{amsmath}
\usepackage{nicefrac}
\usepackage{xspace}
\usepackage{framed}

\usepackage{cite}

\newcommand{\p}[2]{\ensuremath{\frac{\partial #1}{\partial #2 }}}

\begin{document}
    \title{Stabilization of a Multi-Dimensional System of Hyperbolic Balance Laws}
    \author{Michael Herty\footnotemark[1] \and  Ferdinand Thein\footnotemark[1]\, \footnotemark[2]}
    \date{}
    \maketitle
    \begin{abstract}
        We are interested in the feedback stabilization of systems described by Hamilton-Jacobi type equations in $\R^n$.
        A reformulation leads to a stabilization problem for a multi-dimensional system of $n$ hyperbolic partial differential equations.
        Using a novel Lyapunov function taking into account the multi-dimensional geometry we show stabilization in $L^2$ for the arising system using a suitable feedback control.
        We further present examples of such systems partially based on a forming process.
    \end{abstract}
    \renewcommand{\thefootnote}{\fnsymbol{footnote}}
    \footnotetext[1]{IGPM RWTH Aachen, Templergraben 55, D-52056 Aachen, Germany.\\
    \href{mailto:thein@igpm.rwth-aachen.de}{\textit{thein@igpm.rwth-aachen.de}}\\
    \href{mailto:herty@igpm.rwth-aachen.de}{\textit{herty@igpm.rwth-aachen.de}}}
    \footnotetext[2]{Corresponding Author}
    \renewcommand{\thefootnote}{\arabic{footnote}}
    \section{Introduction}

    Stabilization of spatially one--dimensional systems of hyperbolic balance laws has recently attracted research interest in the mathematical and engineering community
    and we refer to the monographs \cite{Bastin2016,MR2302744,MR2655971,MR2412038} for further references and an overview also on related controllability problems.
    Particular applications have been isentropic and isothermal Euler and shallow water equations that form a 2$\times$2 hyperbolic system to model the temporal and spatial evolution of gas and water flow (also on networks).
    Boundary control of such systems has been studied in a series of papers with analytical results  e.g. in the case of gas flow \cite{G1,G2,G3,G4} and water flow \cite{W1,W2,W3,W4,W5,W6}.
    A key tool in the analysis has been the introduction of a Lyapunov function as a weighted $L^2$ (or $H^s$) function that allows to estimate deviations from steady states, see e.g. \cite{Bastin2016}.
    Exponential decay of  Lyapunov function under general {dissipative} conditions has been established for a variety of problem formulations \cite{L1,L2,L4,L5} and a comparisons to other stability concepts is presented
    e.g. in \cite{L7}.
    Stability with respect to a higher $H^s$-norm $(s\geq 2)$ gives stability of the nonlinear system \cite{L5,Bastin2016}.
    Without aiming to give a complete list of extensions, we mention that the results have been recently extended to treat e.g.\ also input-to-state stability \cite{MR2899713},
    numerical discretizations \cite{MR3956429,MR3031137,MR3648349} and nonlocal hyperbolic partial differential equations (PDEs) \cite{MR4172728}.
    \par 
    However, to the best of our knowledge the presented results are limited to spatially one--dimensional case.
    A specific system in two dimensions is discussed in \cite{Dia2013} where a control problem for the shallow water equations is studied.
    There the authors take advantage of the structure of the system and show that the energy is non-increasing upon imposing suited boundary conditions.
    Based on an example in metal forming processes, see \cite{CDC2022} and Section \ref{sec:app_hamjac}, we aim to extend results to multi--dimensional hyperbolic balance laws.
    For the particular structure arising e.g. in the reformulation of Hamilton--Jacobi equations we show $L^2-$stabilization of multi--dimensional hyperbolic systems of balance laws with variable coefficient matrices.
    These matrices are diagonal and hence symmetric matrices. While this restricts the generality it is worth mentioning that many physically relevant systems enjoy this symmetry, see e.g. \cite{Friedrichs1954,Godunov1961}.
    Also, well--posedness results related to symmetric hyperbolic systems are obtained in several publications and without assuming completeness we exemplary refer \cite{Brenner1966} for results in $L^p$
    and \cite{Kato1975} for the linear and quasi linear case. Initial boundary value problems have been studied e.g. in \cite{Majda1975} or \cite{Peyser1975} in the case of constant coefficients.
    In more recent publications, symmetric hyperbolic systems have been used in the modeling and simulation of continuum mechanics \cite{Dumbser2016,Peshkov2021,Busto2022}.
    For  additional references on this topic we  also refer to \cite{Dafermos2016,Ruggeri2021}.
    \par
    In this work we will study the boundary stabilization of a class of multi--dimensional linear hyperbolic balance laws using an extension of the Lyapunov function introduced in \cite{MR2302744}.
    The particular structure of the system is exploited to derive a condition on the feedback law, such that the Lyapunov function decays exponentially fast,
    see Section \ref{sec:defthm} and Section \ref{sec:prf} for the proof.
    The studied system is motivated by applications based on general Hamilton--Jacobi equations and their relation to the studied system will be stated in Section \ref{sec:app_hamjac}.
    Further examples are also given therein.
    \section{Stabilization of Multi-Dimensional Linear Hyperbolic Balance Laws}\label{sec:defthm}
    We are interested in the following initial boundary value problem (IBVP) for the given system of PDEs
    \begin{align}
        \begin{dcases}
        \p{}{t}\mathbf{w} + \sum_{k=1}^d\mathbf{A}^{(k)}(\mathbf{x})\p{}{x_k}\mathbf{w}(t,\mathbf{x}) + \mathbf{B}(\mathbf{x})\mathbf{w}(t,\mathbf{x}) &= 0,\;(t,\mathbf{x})\in[0,T)\times\Omega\\
        \mathbf{w}(0,\mathbf{x}) &= \mathbf{w}_0(\mathbf{x}),\;\mathbf{x}\in\Omega,\\
        \mathbf{w}(t,\mathbf{x}) &= \mathbf{w}_{BC}(t,\mathbf{x}),\;(t,\mathbf{x})\in[0,T)\times\partial\Omega
        \end{dcases}\label{eq:hyp_cons_sys}
    \end{align}
    Here $\mathbf{w} \equiv (w_1(t,\mathbf{x}),\dots,w_n(t,\mathbf{x}))^T$ is the vector of unknowns
    and $\Omega \subset \R^d$ a bounded domain with sufficiently smooth boundary $\partial\Omega$.
    Moreover, $\mathbf{A}^{(k)}$ and $\mathbf{B}$ are sufficiently smooth and bounded $n\times n$ real matrices.\\
    Additionally, the matrices $\mathbf{A}^{(k)}(\mathbf{x}) = (a_{ii}^{(k)}(\mathbf{x}))_{i=1,\dots,n}$ are assumed to be  diagonal matrices and hence have a full set of eigenvectors. The system is hyperbolic.
    With the convention $t \equiv x_0$ and $\mathbf{A}_0 := \mathbf{Id}$ a compressed form of system \eqref{eq:hyp_cons_sys} is 
    \begin{align}
        \sum_{k=0}^d\mathbf{A}^{(k)}(\mathbf{x})\p{}{x_k}\mathbf{w} + \mathbf{B}(\mathbf{x})\mathbf{w} = 0.\label{eq:w_PDE_v2}
    \end{align}
    This system is symmetric hyperbolic in the sense of Friedrichs, i.e., all matrices $\mathbf{A}^{(k)}$ are symmetric and $\mathbf{A}_0 = \mathbf{Id} > 0$, see \cite{Friedrichs1954}.
    The boundary condition will be made more precise later on.
    According to the references \cite{Kato1975,Majda1975,Ruggeri2021,Dafermos2016} there exists a solution $\mathbf{w} \in C^1((0,T),H^s(\Omega))^n$ $,s \geq 1 + \frac{d}{2}$
    if the initial and boundary data are sufficiently smooth.
    In order to study the Lyapunov function it is beneficial to rewrite the equation for each unknown $w_i$ as follows
    \begin{align*}
        &\p{}{t}w_i(t,\mathbf{x}) + \mathbf{a}_i(\mathbf{x})\cdot\nabla w_i(t,\mathbf{x}) + \mathbf{b}_i(\mathbf{x})\cdot\mathbf{w}(t,\mathbf{x}) = 0,\\
        \text{with}\quad &\mathbf{a}_i := (a_{ii}^{(1)},\dots,a_{ii}^{(d)}), \;   	\text{and}\quad \mathbf{b}_i := (b_{i1},\dots,b_{in}).
    \end{align*}
    We assume the following non-characteristic condition
    \begin{align}
       \mathbf{a}_i(\mathbf{x}) \neq \mathbf{0},\,i=1,\dots n.\label{non_charact_cond}
    \end{align}
    In the case $\mathbf{a}_i(\mathbf{x}) = \mathbf{0}$ the PDE for $w_i$ would degenerate to an ODE and thus this case is of no interest in the present work.
    For later use, let
    \begin{align*}
        \mathcal{E}(\mu(\mathbf{x})) &:= \textup{diag}(\exp(\mu_1(\mathbf{x})),\dots,\exp(\mu_n(\mathbf{x}))),\\
        \mathcal{A}(t,\mathbf{x}) &:= \left(\mathbf{w}^T\mathbf{A}^{(1)}\mathcal{E}\mathbf{w},\dots,\mathbf{w}^T\mathbf{A}^{(d)}\mathcal{E}\mathbf{w}\right)^T,\\
        \mathcal{M}^{(k)}(\mathbf{x}) &:= \textup{diag}\left(\p{}{x_k}\mu_1(\mathbf{x}),\dots,\p{}{x_k}\mu_n(\mathbf{x})\right),\,k=1,\dots,d.
    \end{align*}
    The functions $\mu_1(\mathbf{x}),\dots,\mu_n(\mathbf{x})$ will be defined in the subsequent theorem by \eqref{eq:mu_pde}.
    Further, we assume that there exists  a  diagonal matrix $\mathcal{D}$ such that we have
    \begin{align}\label{ass1}
        -\mathbf{v}^T\left(\mathbf{B}^T\mathcal{E} + \mathcal{E}\mathbf{B}\right)\mathbf{v} \leq \mathbf{v}^T\mathcal{D}\mathcal{E}\mathbf{v},\;\forall \mathbf{v} \in \R^n.
    \end{align}
    This kind of dissipativity condition for the linear source term ensures together with \eqref{non_charact_cond} that we find suited weights $\mu_i$ using \eqref{eq:mu_pde}.
    There are different possible choices for $\mathcal{D}$ depending on the problem under consideration.
    We will comment on this later on, see Remark \ref{rem:matrix_est} and in Section \ref{sec:app_hamjac} for some examples.
    The boundary $\partial\Omega$ will be separated in the controllable and uncontrollable part, i.e. for $i=1,\dots,n$
    \begin{align*}
        \Gamma_i^+ &:= \left\{\left.\mathbf{x} \in \partial\Omega\,\right|\,\mathbf{a}_i\cdot\mathbf{n} \geq 0\right\},\\
        \Gamma_i^- &:= \left\{\left.\mathbf{x} \in \partial\Omega\,\right|\,\mathbf{a}_i\cdot\mathbf{n} < 0\right\}.
    \end{align*}
    The part $\Gamma_i^-$ is then split disjointly into the part $\mathcal{C}_i$ where we apply a feedback control $u_i$ and a part $\mathcal{Z}_i$ where we prescribe $w_i = 0$.\\
    Finally we want to specify the notion of exponential stability. We call a solution $\mathbf{w} \in C^1(0,T;H^s(\Omega))^n$ for $s\geq 1+ \frac{d}2$  of the system
    \eqref{eq:hyp_cons_sys} exponentially stable in the $L^2-$sense, iff
    \begin{align}
        \|\mathbf{w}(t,.)\|_{L^2(\Omega)} \leq  \exp(-Ct)\|\mathbf{w}(0,.)\|_{L^2(\Omega)},\,C\in\R_{>0}.\label{def:asymp_stable}
    \end{align}
    The Lyapunov function we will introduce in Theorem \ref{thm:main} is equivalent to the $L^2(\Omega)$ norm.
    Now, we  state the main result of this paper.
    \begin{thm}\label{thm:main}
        Assume \eqref{ass1} and let $\mathbf{w}(t,\mathbf{x}) \in C^1\left((0,T),H^s(\Omega)\right)^n$, $s \geq 1+ d/2,$
        be a solution to the IBVP \eqref{eq:hyp_cons_sys}. Define the Lyapunov function by
        \begin{align}
            L(t) = \int_\Omega \mathbf{w}(t,\mathbf{x})^T\mathcal{E}(\mu(\mathbf{x}))\mathbf{w}(t,\mathbf{x})\,\dd\mathbf{x}\label{eq:lyapunov_gen}
        \end{align}
        and assume that there exists $\mu_i(\mathbf{x}) \in H^s(\Omega)$ such that 
        \begin{align}
            \begin{split}\label{eq:mu_pde}
                \sum_{k=1}^d\left(\mathcal{M}^{(k)}\mathbf{A}^{(k)} + \p{}{x_k}\mathbf{A}^{(k)}\right) + \mathcal{D} &= -\textup{diag}\left(C^{(i)}_L\right)\\
                \text{or written rowwise}\quad\mathbf{a}_i\cdot\nabla\mu_i(\mathbf{x}) + \nabla\cdot\mathbf{a}_i + \mathcal{D}_{ii} &= -C_L^{(i)},\;C^{(i)}_L \geq C_L > 0
            \end{split}
        \end{align}
        holds true for some value $C_L \in \R_{>0}$.
        The boundary condition for \eqref{eq:hyp_cons_sys} is given by
        \begin{align*}
            \mathbf{w}_{BC,i}(t,\mathbf{x}) = \begin{cases} 0 &,\,\mathbf{x}\in\mathcal{Z}_i\\ u_i(t,\mathbf{x}) &,\,\mathbf{x}\in\mathcal{C}_i\end{cases}\;t\in[0,T),\;i=1,\dots,n
        \end{align*}
        where we assume that for $i \in \{1,\dots,n\}$ the feedback controls $u_i$ satisfy
        \begin{align}
            -\sum_{i=1}^n\int_{\mathcal{C}_i}u_i(t,\mathbf{x})^2\left(\mathbf{a}_i\cdot\mathbf{n}\right)\exp(\mu_i(\mathbf{x}))\,\dd\mathbf{x}
            \leq \sum_{i=1}^n\int_{\Gamma_i^+}w_i^2\left(\mathbf{a}_i\cdot\mathbf{n}\right)\exp(\mu_i(\mathbf{x}))\,\dd\mathbf{x}.\label{ineq:boundary_control}
        \end{align}
        Then, the Lyapunov function satisfies
        \begin{align*}
            \frac{\dd}{\dd t}L(t) \leq -C_LL(t).
        \end{align*}
        Further, the solution $\mathbf{w}$ of the IBVP \eqref{eq:hyp_cons_sys} is exponentially stable in the $L^2-$sense \eqref{def:asymp_stable}. 
    \end{thm}
    The proof is given in the following section for better readability. We refer to Remark \ref{rem:control} for conditions under which  $u_i$ exists. In Remark \ref{rem:calc_mu} we comment on the
    solvability of the equation \eqref{eq:mu_pde} and in Remark \ref{rem:matrix_est} for conditions on the existence of $\mathcal{D}$ in assumption \eqref{ass1}. Note that the choice of $s$ in the
    space $H^s$ guarantees that $\mu$ and $\mathbf{w}$ are differentiable in $\mathbf{x}$ by Sobolev embedding.  
    \section{Proof of  Theorem \ref{thm:main}}\label{sec:prf}
    \begin{proof}
        We want to show that the Lyapunov function \eqref{eq:lyapunov_gen} decays exponentially fast under the given conditions.
        We obtain for the time derivative of the Lyapunov function
        \begin{align*}
            \frac{\dd}{\dd t}L(t) &= \frac{\dd}{\dd t}\int_\Omega \mathbf{w}^T\mathcal{E}\mathbf{w}\,\dd\mathbf{x}
            = \int_\Omega \p{}{t}\mathbf{w}^T\mathcal{E}\mathbf{w} + \mathbf{w}^T\mathcal{E}\p{}{t}\mathbf{w}\,\dd\mathbf{x}.
        \end{align*}
        Replacing the time derivative of $\mathbf{w}$ using the PDE system \eqref{eq:hyp_cons_sys} gives
        \begin{align*}
            \frac{\dd}{\dd t}L(t) &= \int_\Omega \left[-\sum_{k=1}^d\mathbf{A}^{(k)}\p{}{x_k}\mathbf{w} - \mathbf{B}\mathbf{w}\right]^T\mathcal{E}\mathbf{w}
            + \mathbf{w}^T\mathcal{E}\left[-\sum_{k=1}^d\mathbf{A}^{(k)}\p{}{x_k}\mathbf{w} - \mathbf{B}\mathbf{w}\right]\,\dd\mathbf{x}\\
            &= -\int_\Omega \sum_{k=1}^d\left[\left(\p{}{x_k}\mathbf{w}\right)^T\mathbf{A}^{(k)}\mathcal{E}\mathbf{w}
            + \mathbf{w}^T\mathcal{E}\mathbf{A}^{(k)}\left(\p{}{x_k}\mathbf{w}\right)\right]\\
            &+ \mathbf{w}^T\left(\mathbf{B}^T\mathcal{E} + \mathcal{E}\mathbf{B}\right)\mathbf{w}\,\dd\mathbf{x}
        \end{align*}
        Now with using the product rule and introducing the abbreviation $\mathcal{A}$ we obtain
        \begin{align*}
            \frac{\dd}{\dd t}L(t) &= -\int_\Omega \sum_{k=1}^d\left[\p{}{x_k}\left(\mathbf{w}^T\mathbf{A}^{(k)}\mathcal{E}\mathbf{w}\right)
            - \mathbf{w}^T\left(\p{}{x_k}\left(\mathcal{E}\mathbf{A}^{(k)}\right)\right)\mathbf{w}\right]\\
            &+ \mathbf{w}^T\left(\mathbf{B}^T\mathcal{E} + \mathcal{E}\mathbf{B}\right)\mathbf{w}\,\dd\mathbf{x}\\
            &= -\int_\Omega \nabla\cdot\mathcal{A}\,\dd\mathbf{x}
            + \int_\Omega \mathbf{w}^T\left(\sum_{k=1}^d\p{}{x_k}\left(\mathcal{E}\mathbf{A}^{(k)}\right)\right)\mathbf{w}
            - \mathbf{w}^T\left(\mathbf{B}^T\mathcal{E} + \mathcal{E}\mathbf{B}\right)\mathbf{w}\,\dd\mathbf{x}\\
            &= -\underbrace{\int_{\partial\Omega} \mathcal{A}\cdot\mathbf{n}\,\dd\mathbf{x}}_{=:\mathcal{B}(t)}\\
            &+ \underbrace{\int_\Omega \mathbf{w}^T\left[\sum_{k=1}^d\left(\mathcal{M}^{(k)}\mathbf{A}^{(k)} + \p{}{x_k}\mathbf{A}^{(k)}\right)\mathcal{E}
            - \left(\mathbf{B}^T\mathcal{E} + \mathcal{E}\mathbf{B}\right)\right]\mathbf{w}\,\dd\mathbf{x}}_{=: \mathcal{I}(t)}.
        \end{align*}
        In the following we split the proof into two parts. First we show that the boundary term $\mathcal{B}(t)$ is non-negative under the given conditions.
        In the second part we estimate $\mathcal{I}(t)$ such that we can use Gronwalls Lemma to show the decay of the Lyapunov function.
        \subsection{Estimating the Boundary Integral}\label{sec:BC}
        In the following we want to discuss the boundary term. We have
        \begin{align*}
            \mathcal{A}\cdot\mathbf{n} &= \sum_{k=1}^d\mathbf{w}^T\mathbf{A}^{(k)}\mathcal{E}\mathbf{w}n_k = \sum_{k=1}^d\sum_{i=1}^n w_ia_{ii}^{(k)}\exp(\mu_i(\mathbf{x}))w_in_k\\
            &= \sum_{i=1}^n w_i^2\exp(\mu_i(\mathbf{x}))\sum_{k=1}^d a_{ii}^{(k)}n_k = \sum_{i=1}^nw_i^2\exp(\mu_i(\mathbf{x}))\left(\mathbf{a}_i\cdot\mathbf{n}\right)
        \end{align*}
        and thus
        \begin{align*}
            \mathcal{B}(t) = \int_{\partial\Omega} \mathcal{A}\cdot\mathbf{n}\,\dd\mathbf{x}
            = \sum_{i=1}^n\int_{\partial\Omega}w_i^2\exp(\mu_i(\mathbf{x}))\left(\mathbf{a}_i\cdot\mathbf{n}\right)\,\dd\mathbf{x}.
        \end{align*}
        Using the partitioning of the boundary introduced before we rewrite the boundary term and obtain
        \begin{align*}
            \mathcal{B}(t) &= \int_{\partial\Omega} \mathcal{A}\cdot\mathbf{n}\,\dd\mathbf{x}\\
            &= \sum_{i=1}^n\left[\int_{\Gamma_i^+}w_i^2\exp(\mu_i(\mathbf{x}))\left(\mathbf{a}_i\cdot\mathbf{n}\right)\,\dd\mathbf{x}
            + \int_{\Gamma_i^-}w_i^2\exp(\mu_i(\mathbf{x}))\left(\mathbf{a}_i\cdot\mathbf{n}\right)\,\dd\mathbf{x}\right]\\
            &= \sum_{i=1}^n\left[\int_{\Gamma_i^+}w_i^2\exp(\mu_i(\mathbf{x}))\left(\mathbf{a}_i\cdot\mathbf{n}\right)\,\dd\mathbf{x}
            + \int_{\mathcal{C}_i}u_i^2\exp(\mu_i(\mathbf{x}))\left(\mathbf{a}_i\cdot\mathbf{n}\right)\,\dd\mathbf{x}\right]\\
            &\geq 0.
        \end{align*}
        \subsection{Estimating the Volume Integral}
        We now want to study the expression
        \begin{align*}
            \mathcal{I}(t) := \int_\Omega \mathbf{w}^T\left[\sum_{k=1}^d\left(\mathcal{M}^{(k)}\mathbf{A}^{(k)} + \p{}{x_k}\mathbf{A}^{(k)}\right)\mathcal{E}
            - \left(\mathbf{B}^T\mathcal{E} + \mathcal{E}\mathbf{B}\right)\right]\mathbf{w}\,\dd\mathbf{x}.
        \end{align*}
        Since $\mathcal{M}^{(k)}$ and $\mathbf{A}^{(k)}$ are diagonal matrices we need to discuss the second term, i.e.
        \begin{align*}
            \mathcal{S}(t) &:= -\int_\Omega \mathbf{w}^T\left(\mathbf{B}^T\mathcal{E} + \mathcal{E}\mathbf{B}\right)
            \mathbf{w}\,\dd\mathbf{x}
        \end{align*}
        where the integrand is a symmetric quadratic form. By assumption \eqref{ass1} there exists a diagonal matrix $\mathcal{D}$ such that
        \begin{align*}
            -\mathbf{w}^T\left(\mathbf{B}^T\mathcal{E} + \mathcal{E}\mathbf{B}\right)\mathbf{w} \leq \mathbf{w}^T\mathcal{D}\mathcal{E}\mathbf{w}
        \end{align*}
        holds and thus we have
        \begin{align*}
            \mathcal{S}(t) &= -\int_\Omega \mathbf{w}^T\left(\mathbf{B}^T\mathcal{E} + \mathcal{E}\mathbf{B}\right)\mathbf{w}\,\dd\mathbf{x}
            \leq \int_\Omega\mathbf{w}^T\mathcal{D}\mathcal{E}\mathbf{w}\,\dd\mathbf{x}.
        \end{align*}
        Altogether we hence obtain
        \begin{align*}
            \mathcal{I}(t) &= \int_\Omega \mathbf{w}^T\left[\sum_{k=1}^d\left(\mathcal{M}^{(k)}\mathbf{A}^{(k)} + \p{}{x_k}\mathbf{A}^{(k)}\right)\mathcal{E}\right]\mathbf{w} + \mathcal{S}(t)\\
            &\leq \int_\Omega \mathbf{w}^T\left[\sum_{k=1}^d\left(\mathcal{M}^{(k)}\mathbf{A}^{(k)} + \p{}{x_k}\mathbf{A}^{(k)}\right)\mathcal{E}\right]\mathbf{w}
            + \int_\Omega\mathbf{w}^T\mathcal{D}\mathcal{E}\mathbf{w}\,\dd\mathbf{x}.
        \end{align*}
        Since the $\mu_i$ satisfy \eqref{eq:mu_pde} we further yield
        \begin{align*}
            \mathcal{I}(t)
            &\leq \int_\Omega \mathbf{w}^T\left[\sum_{k=1}^d\left(\mathcal{M}^{(k)}\mathbf{A}^{(k)} + \p{}{x_k}\mathbf{A}^{(k)}\right) + \mathcal{D}\right]\mathcal{E}\mathbf{w}\,\dd\mathbf{x}\\
            &= -\int_\Omega\mathbf{w}^T\textup{diag}\left(C_L^{(i)}\right)\mathcal{E}\mathbf{w}\,\dd\mathbf{x} = -\sum_{i=1}^nC_L^{(i)}\int_\Omega w_i^2\exp(\mu_i(\mathbf{x}))\,\dd\mathbf{x}.
        \end{align*}
        Altogether we yield for the Lyapunov function
        \begin{align*}
            \frac{\dd}{\dd t}L(t) = -\mathcal{B}(t) + \mathcal{I}(t) \leq \mathcal{I}(t) \leq -\sum_{i=1}^nC_L^{(i)}\int_\Omega w_i^2\exp(\mu_i(\mathbf{x}))\,\dd\mathbf{x}
            \leq -C_LL(t).
        \end{align*}
        Applying Gronwalls Lemma gives the claimed exponential decay and thus the proof is complete.
    \end{proof}
    Concerning the extension to the $H^s$ norm we have the following observation.
    \begin{rem}
        It is a question of interest whether the obtained result can be extended to the exponential stability in the corresponding $H^s$-norm.
        The basic strategy is outlined in \cite{Bastin2016} (Chap. 4, pp. 153).
        The Lyapunov function can be defined by
        \begin{align}
            L(t) := \sum_{|\alpha|\leq s}\left\|\mathcal{E}\left(\mu^{(\alpha)}\right)^\frac{1}{2}\partial^\alpha \mathbf{u}\right\|^2_2.\label{def:Hs_Lyapunov}
        \end{align}
        Note that by introducing the $\mathcal{E}\left(\mu^{(\alpha)}\right)$ we want to highlight that in principal it is possible to choose suited weight functions for every occurring multi-index $|\alpha| \leq s$.
        Once the Lyapunov function is introduced the strategy is the same since due to the governing PDE \eqref{eq:hyp_cons_sys} the new system obtained by differentiation has the same structure.
        However, one has to make sure that the new terms fulfill the desired assumptions made for the main theorem. Once this verified the arguments are similar.
    \end{rem}
    \begin{rem}
        Theorem \ref{thm:main} is formulated for solutions being in $H^s$ and thus providing enough smoothness for the steps thorughout the  proof.
        However, it is pointed out in \cite{Bastin2016} (Chap. 2, p. 64) and \cite{Bastin2019} (Lem. 4.2) that due to the density of $C^1$ functions in $L^2$ it is possible to transfer the obtained results to more
        general weak solutions in $L^2$.
    \end{rem}
    In the following we want to discuss different special cases which lead to mathematical simplifications or situations of particular interest.  
    \begin{rem}\label{rem:control}
        The inequality for the control $u_i$ can be simplified under additional assumptions.
        \begin{enumerate}[(i)]
        \item A possible simplifying assumption on the controls is that they are uniform in space, i.e.\ $u_i = u_i(t)$.
        This leads to the following condition for the controls
        \begin{align*}
            -\sum_{i=1}^nu_i(t)^2\int_{\mathcal{C}_i}\left(\mathbf{a}_i\cdot\mathbf{n}\right)\exp(\mu_i(\mathbf{x}))\,\dd\mathbf{x}
            \leq \sum_{i=1}^n\int_{\Gamma_i^+}w_i^2\left(\mathbf{a}_i\cdot\mathbf{n}\right)\exp(\mu_i(\mathbf{x}))\,\dd\mathbf{x}.
        \end{align*}
        \item Further the same control may be applied to all components which leads to $u_i \equiv u(t,\mathbf{x}),\,i=1,\dots,n$.
        This leads to the following condition for the controls
        \begin{align*}
            -\sum_{i=1}^n\int_{\mathcal{C}_i}u(t,\mathbf{x})^2\left(\mathbf{a}_i\cdot\mathbf{n}\right)\exp(\mu_i(\mathbf{x}))\,\dd\mathbf{x}
            \leq \sum_{i=1}^n\int_{\Gamma_i^+}w_i^2\left(\mathbf{a}_i\cdot\mathbf{n}\right)\exp(\mu_i(\mathbf{x}))\,\dd\mathbf{x}.
        \end{align*}
        \item Assume that both previous assumptions hold, i.e. $u_i \equiv u(t),\,i=1,\dots,n$. Then we can explicitly give a condition for the feedback control as
        \begin{align}
            u(t)^2 \leq -\left(\sum_{i=1}^n\int_{\mathcal{C}_i}\left(\mathbf{a}_i\cdot\mathbf{n}\right)\exp(\mu_i(\mathbf{x}))\,\dd\mathbf{x}\right)^{-1}
            \sum_{i=1}^n\int_{\Gamma_i^+}w_i^2\left(\mathbf{a}_i\cdot\mathbf{n}\right)\exp(\mu_i(\mathbf{x}))\,\dd\mathbf{x}.\label{ineq:boundary_control2}
        \end{align}
        \end{enumerate}
    \end{rem}
    \begin{rem}\label{rem:matrix_est}
        We comment on the choice of the matrix $\mathcal{D}$.
        \begin{enumerate}[(i)]
            \item In the general case the following estimate always holds
            \begin{align*}
                -\mathbf{w}^T\left(\mathbf{B}^T\mathcal{E} + \mathcal{E}\mathbf{B}\right)\mathbf{w} &= -2\mathbf{w}^T\mathbf{B}^T\underbrace{\mathcal{E}}_{>0}\mathbf{w}
                =
                -2\underbrace{\mathbf{w}^T\mathcal{E}^{\frac{1}{2}}}_{:=\mathbf{y}^T}\underbrace{\mathcal{E}^{-\frac{1}{2}}\mathbf{B}^T\mathcal{E}^{\frac{1}{2}}}_{=:\mathbf{Q}}\underbrace{\mathcal{E}^{\frac{1}{2}}\mathbf{w}}_{=:\mathbf{y}}\\
                &\leq C(\mathbf{Q})\mathbf{y}^T\mathbf{y} = C(\mathbf{Q})\mathbf{w}^T\mathcal{E}\mathbf{w}.
            \end{align*}
            The constant $C(\mathbf{Q})$ depends on the minimal eigenvalue of $\mathbf{Q} + \mathbf{Q}^T$.
            \item If $\mathbf{B}$ is symmetric we can improve the estimate using the minimal eigenvalue $\lambda_{min} \in Spec(\mathbf{B}) \subset \R$
            \begin{align*}
                -\mathbf{w}^T\left(\mathbf{B}^T\mathcal{E} + \mathcal{E}\mathbf{B}\right)\mathbf{w} = -2\mathbf{w}^T\mathbf{B}^T\mathcal{E}\mathbf{w}
                \leq -2\lambda_{min}\mathbf{w}^T\mathcal{E}\mathbf{w}.
            \end{align*}
            This corresponds to the matrix $\mathcal{D} = -2\lambda_{min}\textup{\textbf{Id}}$.
            \item In the case that $\mathbf{B}$ is a diagonal matrix itself we simply have $\mathcal{D} = -2\mathbf{B}$ without  any assumption.
            \item If the quadratic form is positive definite,  the source term is estimated  by $\mathcal{S}(t) \leq 0$ and thus equation \eqref{eq:mu_pde} simplifies accordingly.
        \end{enumerate}
    \end{rem}
    Case (iii) of Remark \ref{rem:matrix_est} is a special case of interest since it leads to the following result.
    \begin{cor}\label{cor:sharp_case}
        Let the assumption of Theorem \eqref{thm:main} hold true. Additionally assume that  $\mathbf{B}$ is a diagonal matrix. The Lyapunov function is given as
        \begin{align*}
            L(t) = \int_\Omega \mathbf{w}(t,\mathbf{x})^T\mathcal{E}(\mu(\mathbf{x}))\mathbf{w}(t,\mathbf{x})\,\dd\mathbf{x}
        \end{align*}
        where we assume that there exists $\mu_i(\mathbf{x}) \in H^s(\Omega)$ satisfy
        \begin{align}
            \begin{split}\label{eq:mu_pde2}
                \sum_{k=1}^d\left(\mathcal{M}^{(k)}\mathbf{A}^{(k)} + \p{}{x_k}\mathbf{A}^{(k)}\right) - 2\mathbf{B} &= -C_L\mathbf{Id}\\
                \text{or written rowwise}\quad\mathbf{a}_i\cdot\nabla\mu_i(\mathbf{x}) + \nabla\cdot\mathbf{a}_i - 2\mathbf{B}_{ii} &= -C_L,\; C_L > 0
            \end{split}
        \end{align}
        for some value $C_L \in \R_{>0}$.
        The boundary condition for \eqref{eq:hyp_cons_sys} is given by
        \begin{align*}
            \mathbf{w}_{BC,i}(t,\mathbf{x}) = \begin{cases} 0 &,\,\mathbf{x}\in\mathcal{Z}_i\\ u_i(t,\mathbf{x}) &,\,\mathbf{x}\in\mathcal{C}_i\end{cases}\;t\in[0,T),\;i=1,\dots,n
        \end{align*}
        where the $u_i$ satisfies
        \begin{align}
            -\sum_{i=1}^n\int_{\mathcal{C}_i}u_i(t,\mathbf{x})^2\left(\mathbf{a}_i\cdot\mathbf{n}\right)\exp(\mu_i(\mathbf{x}))\,\dd\mathbf{x}
            = \sum_{i=1}^n\int_{\Gamma_i^+}w_i^2\left(\mathbf{a}_i\cdot\mathbf{n}\right)\exp(\mu_i(\mathbf{x}))\,\dd\mathbf{x}.\label{eq:boundary_control}
        \end{align}
        Then the Lyapunov function satisfies
        \begin{align*}
            L(t) = L(0)\exp(-C_Lt).
        \end{align*}
        and thus the solution $\mathbf{w}$ of the IBVP \eqref{eq:hyp_cons_sys} is asymptotically stable in the $L^2-$sense.
    \end{cor}
    \begin{proof}
        The proof basically follows the same steps as the previous one. With the proposed choice of the control we obtain $\mathcal{B}(t) = 0$.
        For the volume integral $\mathcal{I}(t)$ we have
        \begin{align*}
            \mathcal{I}(t) &= \int_\Omega \mathbf{w}^T\left[\sum_{k=1}^d\left(\mathcal{M}^{(k)}\mathbf{A}^{(k)} + \p{}{x_k}\mathbf{A}^{(k)}\right)\mathcal{E}
            - \left(\mathbf{B}^T\mathcal{E} + \mathcal{E}\mathbf{B}\right)\right]\mathbf{w}\,\dd\mathbf{x}\\
            &= \int_\Omega \mathbf{w}^T\left[\sum_{k=1}^d\left(\mathcal{M}^{(k)}\mathbf{A}^{(k)} + \p{}{x_k}\mathbf{A}^{(k)}\right)\mathcal{E}
            - 2\mathbf{B}\mathcal{E}\right]\mathbf{w}\,\dd\mathbf{x}\\
            &= -\int_\Omega\mathbf{w}^TC_L\mathcal{E}\mathbf{w}\,\dd\mathbf{x} = -C_L\sum_{i=1}^n\int_\Omega w_i^2\exp(\mu_i(\mathbf{x}))\,\dd\mathbf{x}.
        \end{align*}
        Together this gives
        \begin{align*}
            \frac{\dd}{\dd t}L(t) = \mathcal{B}(t) + \mathcal{I}(t) = \mathcal{I}(t) = -C_L\sum_{i=1}^n\int_\Omega w_i^2\exp(\mu_i(\mathbf{x}))\,\dd\mathbf{x}
            = -C_LL(t).
        \end{align*}
    \end{proof}
    \begin{rem}\label{rem:calc_mu}
        An important task that remains is the calculation of the weights $\mu_i$ which have to satisfy \eqref{eq:mu_pde}
        \[
        \mathbf{a}_i\cdot\nabla\mu_i(\mathbf{x}) + \nabla\cdot\mathbf{a}_i + \mathcal{D}_{ii} = -C_L^{(i)}.
        \]
        Note that this PDE is an inhomogeneous linear first order transport PDE in $d$ dimensions with variable coefficients.
        A classical approach to obtain a solution is the method of characteristics as presented e.g. \cite{Evans1998}.
        However, we want to emphasize that although we are considering a control problem on a bounded domain, the weights $\mu_i$ can be determined from \eqref{eq:mu_pde} without any additional restrictions.
        Thus it is worth noting that only the PDE has to be satisfied and there is freedom in the choice of further conditions which may be due to constraints from the application.
        Moreover further simplifications are possible and will be presented using  particular examples in Section \ref{sec:app_hamjac}.
    \end{rem}
    \section{Application to Stabilization of Hamilton--Jacobi Equations}\label{sec:app_hamjac}
    In this section we discuss the application of the stabilization result.
    In particular, we show conditions for the solvability of equation \eqref{eq:mu_pde} as well as on the particular structure of the matrices $\mathbf{A}^{(k)}.$
    Theorem \eqref{thm:main} has been motivated by a stabilization problem arising in a forming process leading to the question of  stabilization of Hamilton--Jacobi equations.
    We start by exploring formally the  relation to this class of PDEs.
    Hamilton-Jacobi equations for describing the level--set of a deforming process have been used e.g. in \cite{CDC2022}.
    Here, a Titanium deformation process needed to be stabilized at a desired level--set by applying an external force at the boundary. The quantity of interest
    \[
    \phi = \phi(t,\mathbf{x}),\;\phi : \Omega \to \R
    \]
    for $\Omega \in \mathbb{R}^n$ has been used to describe the boundary between formed and unformed material. The evolution of this set  $\{ (t,\mathbf{x}):  \phi(t,\mathbf{x}) = 0\}$ is mathematically described by a Hamilton-Jacobi equation
    \begin{align}
        \begin{dcases}
            \p{\phi}{t} + H(\mathbf{x}, \nabla \phi) &= 0,\;(t,\mathbf{x})\in(0,T)\times\Omega,\\
            \phi(0,\mathbf{x}) &= \phi_0(\mathbf{x}),\;\mathbf{x}\in\Omega,\\
            \nabla\phi(t,\mathbf{x})\cdot\mathbf{n} &= 0,\;(t,\mathbf{x})\in[0,T)\times\partial\Omega
        \end{dcases}\label{eq:hamilton_pde}
    \end{align}
    for an Hamiltonian $H$ that describes the flow field induced by the process. If we introduce the new variable $\mathbf{\Psi}(t,\mathbf{x}) = \nabla\phi(t,\mathbf{x}),$ i.e.,
    \[
    \psi_i := \partial_{x_i}\phi.
    \]
    we formally transform the initial equation into a system of (hyperbolic) PDEs. The system \eqref{eq:hamilton_pde} reads then  
    \begin{align}
        \begin{dcases}
            \p{}{t}\mathbf{\Psi} + \nabla (H(\mathbf{x}, \mathbf{\Psi})) &= 0,\;(t,\mathbf{x})\in(0,T)\times\Omega\\
            \mathbf{\Psi}(0,\mathbf{x}) &= \nabla\phi_0(\mathbf{x}),\;\mathbf{x}\in\Omega,\\
            \mathbf{\Psi}(t,\mathbf{x}) &= \mathbf{\Psi}_{BC}(t,\mathbf{x}),\;(t,\mathbf{x})\in[0,T)\times\partial\Omega
        \end{dcases}\label{eq:hyp_cons_sys_v2}
    \end{align}
    Note that von-Neumann boundary conditions in \eqref{eq:hamilton_pde} are now  Dirichlet  boundary condition in \eqref{eq:hyp_cons_sys_v2} understood e.g. in the sense of \cite{Bardos}.
    In view of the stabilization of the process we then  linearize $H$ at a desired reference state $\overline{\mathbf{\Psi}}$ and obtain
    \begin{align}
        H(\mathbf{x},\mathbf{\Psi}) \approx H(\mathbf{x},\overline{\mathbf{\Psi}}) + \nabla_{\mathbf{\Psi}}H(\mathbf{x},\overline{\mathbf{\Psi}})\cdot\left(\mathbf{\Psi} - \overline{\mathbf{\Psi}}\right).
        \label{approx:H}
    \end{align}
    Denote by  $\mathbf{\overline{\Psi}} = \nabla\overline{\phi}$ and assume $\overline{\phi}$ fulfills  \eqref{eq:hyp_cons_sys_v2}.
    Deviations of this desired state are considered and hence  $\mathbf{\Psi}$ is given by 
    \[
    \mathbf{\Psi} = \mathbf{\overline{\Psi}} + \mathbf{w},
    \]
    for some initial disturbance we assume $w_i(0,\mathbf{x}) \neq 0$. By the following computation we obtain  a linear system using \eqref{approx:H} starting from \eqref{eq:hyp_cons_sys_v2}
    \begin{align*}
        0 &= \p{}{t}\mathbf{\Psi} + \nabla (H(\mathbf{x}, \mathbf{\Psi}))\\
        &\approx \p{}{t}\mathbf{\Psi} + \nabla (H(\mathbf{x},\overline{\mathbf{\Psi}})
        + \nabla_{\mathbf{\Psi}}H(\mathbf{x},\overline{\mathbf{\Psi}})\cdot\left(\mathbf{\Psi} - \overline{\mathbf{\Psi}}\right))\\
        &= \p{}{t}(\mathbf{\overline{\Psi}} + \mathbf{w}) + \nabla (H(\mathbf{x},\overline{\mathbf{\Psi}}) + \nabla_{\mathbf{\Psi}}H(\mathbf{x},\overline{\mathbf{\Psi}})\cdot\mathbf{w})\\
        &= \p{}{t}\mathbf{w} + \nabla\left(\nabla_{\mathbf{\Psi}}H(\mathbf{x},\overline{\mathbf{\Psi}})\cdot\mathbf{w}\right).
    \end{align*}
    By  definition of $\mathbf{\Psi}$ and $\mathbf{\overline{\Psi}}$ we have
    \begin{align*}
        &\p{}{x_i}\Psi_j = \p{^2}{x_i\partial x_j}\phi = \p{}{x_j}\Psi_i,\quad\text{and}\quad
        \p{}{x_i}\overline{\Psi}_j = \p{^2}{x_i\partial x_j}\overline{\phi} = \p{}{x_j}\overline{\Psi}_i
    \end{align*}
    which implies
    \begin{align}
        \p{}{x_i}w_j = \p{}{x_j}w_i.\label{commute_deriv_w}
    \end{align}
    Furthermore, we denote
    \[
    H_k(\mathbf{x}) := \left.\p{}{\psi_k}H(\mathbf{x},\mathbf{\Psi})\right|_{\mathbf{\Psi} = \overline{\mathbf{\Psi}}}
    \]
    to obtain the following set of equations for $i=1,\dots,n$ 
    \begin{align*}
        0 = \p{}{t}w_i + \p{}{x_i}\left(\sum_{k=1}^nH_k\cdot w_k\right)
        = \p{}{t}w_i + \sum_{k=1}^n\left[\p{H_k}{x_i}w_k + H_k\p{w_k}{x_i}\right]
    \end{align*}
    Thus by introducing the source term
    \[
    S_i = -\sum_{k=1}^n\p{H_k}{x_i}w_k
    \]
    and using \eqref{commute_deriv_w} to write
    \[
      \p{w_k}{x_i} = \p{w_i}{x_k}
    \]
    we obtain the following system of $n$ (transport) PDEs studied in the Theorem \eqref{thm:main}
    \begin{align}
        \begin{split}\label{eq:w_PDE}
            \p{}{t}w_i(t,\mathbf{x}) + \sum_{k=1}^nH_k(\mathbf{x})\p{w_i}{x_k}(t,\mathbf{x}) &= S_i\\
            \text{or}\quad\p{}{t}w_i(t,\mathbf{x}) + \nabla_{\mathbf{\Psi}}H(\mathbf{x},\overline{\mathbf{\Psi}})\cdot\nabla w_i(t,\mathbf{x}) &= S_i.
        \end{split}
    \end{align}
    Note that the obtained equations \eqref{eq:w_PDE} are linear in $w_i$ since $H$ and the reference state $\overline{\mathbf{\Psi}}$ are assumed to be known.
    Moreover the sources are algebraic for the same reason and are responsible for the coupling of the equations.
    Note that the source term vanishes identically for $\nabla_{\mathbf{\Psi}}H(.,\overline{\mathbf{\Psi}}) = \nabla_{\mathbf{\Psi}}H(\overline{\mathbf{\Psi}}) \in \R^n$, i.e. the gradient is constant.
    This is in fact the case for the actual forming process, see \cite{CDC2022}. However, the presented result extends this towards more general $S_i.$
    \par
    For the proof of the result the compact vector form is suitable. We show here the relation towards $H.$ Let the matrices $$\mathbf{A}^{(k)}(\mathbf{x}) := H_k(\mathbf{x})\mathbf{Id},\; k=1,\dots,n$$ be given and
    \[
    \mathbf{B}(\mathbf{x}) := \left(\p{H_k}{x_i}\right)_{i,k=1,\dots,n}, \; \mathbf{S} = -\mathbf{B}\mathbf{w}. \; 
    \]
    This leads to the system 
    \begin{align*}
        \p{}{t}\mathbf{w} + \sum_{k=1}^n\mathbf{A}^{(k)}(\mathbf{x})\p{}{x_k}\mathbf{w} + \mathbf{B}(\mathbf{x})\mathbf{w} = 0.
    \end{align*}
    The matrices $\mathbf{A}^{(k)}$ are diagonal matrices with the multiple eigenvalue $H_k(\mathbf{x})$ and hence have a full set of eigenvectors.
    By equation \eqref{eq:w_PDE} we  conclude
    \[
    \mathbf{a}_i(\mathbf{x}) = \nabla_{\mathbf{\Psi}}H(\mathbf{x},\overline{\mathbf{\Psi}}).
    \]
    In the following sections we discuss particular cases that may serve as  examples and simplifications. 
    \subsection{Potential Gradient Flow}
    A particular case of interest is a Hamiltonian of the form
    \begin{align*}
        H(\mathbf{x},\mathbf{\Psi}) = \tilde{\mathbf{H}}(\mathbf{x})\cdot\mathbf{\Psi}.
    \end{align*}
    If additionally $\tilde{\mathbf{H}}(\mathbf{x})$ is the gradient of a potential $\varphi(\mathbf{x})$ we have
    \begin{align*}
        \tilde{\mathbf{H}}(\mathbf{x}) &= \nabla_{\mathbf{\Psi}}H(\mathbf{x},\mathbf{\Psi}) = \nabla\varphi(\mathbf{x}),\\
        \mathbf{B}(\mathbf{x}) &= \left(\p{H_j}{x_i}\right)_{i,j=1,\dots,n} = \left(\p{}{x_i}\p{\varphi}{x_j}\right)_{i,j=1,\dots,n} = \mathbf{D}^2\varphi(\mathbf{x}).
    \end{align*}
    Thus $\mathbf{B}$ is the Hessian of a potential $\varphi$ and hence we have the situation of case (ii), Remark \ref{rem:matrix_est}.
    \subsection{Vanishing Mixed Source Terms}\label{subsec:vanish_mix_source}
    Consider the situation of $\mathbf{B}$ being a diagonal matrix and thus we are in the case (iii) of Remark \ref{rem:matrix_est}.
    In this case  the source term is as follows
    \begin{align*}
        \mathcal{S}(t) &= \sum_{i=1}^n\int_\Omega 2w_iS_i\exp(\mu_i(\mathbf{x}))\,\dd\mathbf{x}
        = -2\sum_{i=1}^n\int_\Omega \p{H_i}{x_i}w_i^2\exp(\mu_i(\mathbf{x}))\,\dd\mathbf{x}
    \end{align*}
    Moreover, Corollary \ref{cor:sharp_case} can be applied.
    Thus by choosing each $\mu_i$ according to
    \[
    \nabla_{\mathbf{\Psi}}H\cdot\nabla \mu_i + \nabla\cdot\nabla_{\mathbf{\Psi}}H - 2\p{H_i}{x_i} = -C_L
    \]
    we obtain
    \[
    \mathcal{I}(t) = \sum_{i=1}^n\int_\Omega w_i^2\left(\nabla_{\mathbf{\Psi}}H\cdot\nabla \mu_i + \nabla\cdot\nabla_{\mathbf{\Psi}}H - 2\p{H_i}{x_i}\right) \exp(\mu_i(\mathbf{x}))\,\dd\mathbf{x}
    = -\sum_{i=1}^nC_LL_i(t).
    \]
    If we now also choose the control $u_i$ such that 
    \begin{align*}
    	u(t)^2 = \left(\sum_{i=1}^n\int_{\mathcal{C}_i}\left(\nabla_{\mathbf{\Psi}}H\cdot\mathbf{n}\right)\exp(\mu_i(\mathbf{x}))\,\dd\mathbf{x}\right)^{-1}
    	\sum_{i=1}^n\int_{\Gamma_i^+}w_i^2\left(\nabla_{\mathbf{\Psi}}H\cdot\mathbf{n}\right)\exp(\mu_i(\mathbf{x}))\,\dd\mathbf{x}.
    \end{align*}
    we obtain $\mathcal{B}(t) = 0$ leading altogether to (see Cor. \ref{cor:sharp_case})
    \[
    \frac{\dd}{\dd t}L(t) = -C_LL(t)\quad\Leftrightarrow\quad L(t) = L(0)\exp(-C_Lt).
    \]
    In this particular case we computed exactly the decay of the Lyapunov function which may be beneficial for testing numerical methods.
    An example of a Hamiltonian $H$ leading to this situation is e.g.
    \[
    H(\mathbf{x},\mathbf{\Psi}) = \sum_{i=1}^nH_i(x_i)\psi_i\quad\Rightarrow\quad \p{H_i}{x_j} = 0,\,i\neq j.
    \]
    In this case, the PDE for $\mu_i$ simplifies and we can find a solution as
    \begin{align}
        \mu_i(\mathbf{x}) = \sum_{k=1}^nF^{(i)}_k(x_k)\label{eq:mu_ansatz_sep}
    \end{align}
    for some functions $F^{(i)}_k$ as follows. Note that due to the particular structure of $\mu_i$  we have
    \[
    \p{^2}{x_k\partial x_l}\mu_i \equiv 0,\;k\neq l.
    \]
    By  differentiating the PDE for $\mu_i$ w.r.t.\ $x_k$ we obtain
    \begin{align*}
        0 &= \p{}{x_k}\left[\nabla\cdot\nabla_{\mathbf{\Psi}}H + \nabla_{\mathbf{\Psi}}H\cdot\nabla \mu_i - 2\p{H_i}{x_i}\right]\\
        &= \p{^2}{x_k^2}H_k + \p{}{x_k}H_k\p{}{x_k}\mu_i + H_k\p{^2}{x_k^2}\mu_i - 2\p{^2}{x_k\partial x_i}H_i\delta_{ki}\\
        &= \p{^2}{x_k^2}H_k + \p{}{x_k}H_kF_k'^{(i)}(x_k) + H_kF_k''^{(i)}(x_k) - 2\p{^2}{x_k\partial x_i}H_i\delta_{ki}\\
        \Leftrightarrow\quad &\frac{\dd}{\dd x_k}\left(H_k(x_k)F_k'^{(i)}(x_k)\right) = 2\p{^2}{x_k\partial x_i}H_i\delta_{ki} - \p{^2}{x_k^2}H_k\\
        \Leftrightarrow\quad &H_k(x_k)F_k'^{(i)}(x_k) = \int^{x_k}2\p{^2}{x_k\partial x_i}H_i(s)\delta_{ki} - \p{^2}{x_k^2}H_k(s)\,\dd s.
    \end{align*}
    Hence, the $F_k^{(i)}$are  determined by integration which in turn define $\mu_i.$ 
    \subsection{Constant Hamiltonian Gradient}\label{subsec:const_hamilton}
    Under the assumption that $\nabla_{\mathbf{\Psi}}H = \mathbf{C} \in\R^n\setminus\{\mathbf{0}\}$ is constant in each component the source terms $S_i$ vanish identically.
    In this case the PDE for $\mu_i$ reads
    \begin{align}
        -C_L^{(i)} = \nabla_{\mathbf{\Psi}}H\cdot\nabla\mu_i,\;C^{(i)}_L\in\R_{>0}.\label{eq:pde_mu_v2}
    \end{align}
    This PDE is an inhomogeneous linear advection equation. We choose an index $j$ with $H_j \neq 0$ and consider the following problem
    \begin{align*}
        \begin{dcases}
            \p{}{\tau}\mu_i(\tau,\mathbf{y}) + \mathbf{a}\cdot\nabla_y\mu_i(\tau,\mathbf{y}) = -\frac{C_L^{(i)}}{H_j},\\
            \mu_i(0,\mathbf{y}) = g_i(\mathbf{y})
        \end{dcases}\\
        \text{with}\;\tau := x_j, \; 
        \mathbf{y} := (x_1,\dots,x_{j-1},x_{j+1},\dots,x_n),\\
        \mathbf{a} := \frac{1}{H_j}\overline{\mathbf{H}},\quad 
        \overline{\mathbf{H}} := (H_1,\dots,H_{j-1},H_{j+1},\dots,H_n).
    \end{align*}
    Its general solution is given by
    \[
    \mu_i(\tau,\mathbf{y}) = g_i(\mathbf{y} - \mathbf{a}\tau) + \int_0^\tau f(\mathbf{y} + (s-\tau)\mathbf{a},s)\,\dd s
    \]
    which leads to
    \begin{align*}
        \mu_i(\mathbf{x}) = g_i(\mathbf{y} - \mathbf{a}x_j) - \frac{C_L^{(i)}}{H_j}x_j.
    \end{align*}
    The function $g_i$ can be chosen arbitrarily. Hence, a full solution to equation \eqref{eq:mu_pde} exists. Note that the case $g_i\equiv 0$ corresponds to the 1D case of stabilizing controls as
    introduced e.g. in \cite{Bastin2016,MR2302744}.\\
    Note that the choice for $\mu(x_1,x_2)$ made in \cite{CDC2022} is a combination of the present case with $\nabla_{\mathbf{\Psi}}H = \mathbf{C} \in\R^n\setminus\{\mathbf{0}\}$
    and the solution ansatz $\eqref{eq:mu_ansatz_sep}$ for $\mu$ used in the previous part.
    The excluded case $\nabla_{\mathbf{\Psi}}H = \mathbf{C} = \mathbf{0}$ is of no interest since it implies $H = H(\mathbf{x})$ and thus the Hamilton-Jacobi equation becomes the trivial ODE
    \[
      \p{}{t}\Phi(t,\mathbf{x}) + H(\mathbf{x}) = 0.
    \]
    \section{Summary}
    A novel Lyapunov function for $L^2-$control of a class of multi--dimensional systems of hyperbolic equations has been presented.
    The analysis relies on the fact that this system can be derived from Hamilton--Jacobi equations.
    A stabilizing feedback control has been derived and  exponential decay of a weighted $L^2-$norm has been established.
    Several examples have been presented where some are inspired by a model for forming processes. Future extensions could be given towards numerical discretizations as well as stabilization in $H^s-$norm.
    %
    \section*{Acknowledgments}
    \small{This research is part of the DFG SPP 2183 \emph{Eigenschaftsgeregelte Umformprozesse}, project 424334423 and the authors thank the Deutsche Forschungsgemeinschaft (DFG, German Research Foundation)
    for the financial support through 320021702/GRK2326, 333849990/IRTG-2379, CRC1481, 423615040/SPP1962, 462234017, 461365406, ERS SFDdM035 and under Germany’s Excellence Strategy EXC-2023 Internet of Production 390621612
    and under the Excellence Strategy of the Federal Government and the Länder.}
    \phantomsection
    \bibliographystyle{abbrv}
    \bibliography{stb_hypsys_literature}
\end{document}